\documentclass{article}
\usepackage{graphicx} 

\usepackage[
backend=biber,
style=alphabetic,
sorting=ynt
]{biblatex}
\usepackage[margin=1.5in]{geometry}

\usepackage{amsmath, amsthm, amssymb, dsfont, amsfonts, color, soul, mathtools, enumerate}

\theoremstyle{definition}
\newtheorem{theorem}{Theorem}

\newtheorem{proposition}[theorem]{Proposition}
\newtheorem{lemma}[theorem]{Lemma}

\newtheorem{remark}[theorem]{Remark}

\renewcommand{\bar}{\overline}

\renewcommand{\P}{\mathds{P}}
\newcommand{\E}{\mathds{E}}
\newcommand{\diag}{\text{diag}}

\usepackage{tocloft}
\usepackage[colorlinks=true,linkcolor=blue,citecolor=blue,urlcolor=blue]{hyperref}

\title{Limit Profiles for Separation Distance}
\author{
Peter E. Francis\thanks{Department of Mathematics, Stony Brook University,
Stony Brook, NY 11794. E-mail: {\tt peter.e.francis@stonybrook.edu}. Funded by DMS-2346986.}
\and Evita Nestoridi \thanks{Department of Mathematics, Stony Brook University,
Stony Brook, NY 11794. E-mail: {\tt evrydiki.nestoridi@stonybrook.edu}. Funded by DMS-2346986 and MPS-TSM-00007955.}
}

\begin{document}
\maketitle

\begin{abstract}
This paper studies limit profiles for the separation distance. A limit profile records the limiting shape of the distance to stationarity inside the cutoff window, at times of the form $t_n+cw_n$. We start with two famous card shuffles, a general setup for inverse riffle shuffles and random transpositions, and we determine their separation distance limit profiles. We then develop a spectral comparison technique and study  continuity properties in the style of \cite{Nestoridi, NestoridiContinuity}, adapted to separation distance. The comparison method is illustrated through random transpositions, as well as random walks on product groups and the hypercube.
\end{abstract}

\tableofcontents

\newpage
\section{Introduction}
One of the most effective ways to understand separation distance is through strong stationary times \cite{AldousDiaconisShuffling,AldousDiaconis,DiaconisFillDuality,NestoridiHyperplaneSST}. A strong stationary time $T$ is a stopping time for a Markov chain, whose definition implies that the separation distance is bounded above by the tail $\P[T>t]$; when $T$ is optimal, this upper bound is an equality. Thus, whenever an explicit optimal strong stationary time is available, the cutoff window and the limit profile can often be read directly from the limiting tail of $T$. 

This point of view is especially natural for shuffling examples. The modern mathematical study of riffle shuffling begins with the Gilbert--Shannon--Reeds model. Aldous and Diaconis used strong stationary times to study card shuffling \cite{AldousDiaconisShuffling}, and Bayer and Diaconis proved the classical ``seven shuffles'' theorem for the Gilbert--Shannon--Reeds shuffle \cite{BayerDiaconis1992}. Lalley studied biased $p$-shuffles and identified the cold-spot obstruction \cite{Lalley2000}; Fulman developed the combinatorics of biased riffle shuffles \cite{FulmanBiasedRiffle}; and Sellke later proved cutoff for asymmetric riffle shuffles, confirming Lalley's conjecture \cite{Sellke2022}. Recent work of Sellke, Shi, and Wang proves universality for riffle shuffles with general pile-size distributions, including uniformly random cuts and exact bisections \cite{SellkeShiWang2025}. In the inverse riffle-type shuffles studied below, each card accumulates a selection history, and the relevant optimal strong stationary time is the first time all card histories are distinct. Our first contribution is to characterize the separation limit profile for a general class of inverse riffle shuffles.

Let $P$ be the transition matrix of an irreducible Markov chain on a finite state space $X$ with stationary measure $\pi$. The separation distance is
$$s(t):=\max_{x\in X}s_x(t),\qquad s_x(t):=\max_{y\in X}\left\{1-\frac{P^t(x,y)}{\pi(y)}\right\}.$$
It is stronger than total variation distance in the sense that $d(t)\leq s(t)$, where
$$d(t):=\max_{x\in X}\left\{\frac12\sum_{y\in X}|P^t(x,y)-\pi(y)|\right\}.$$
For reversible chains one also has $s(2t)\leq4d(t)$, but total variation cutoff and separation cutoff are not equivalent in general \cite{Jonathan}. In this paper, all cutoff and limit-profile statements are made with respect to $s(t)$.

A sequence of chains exhibits separation cutoff at times $t_n$ with window $w_n$ if $w_n=o(t_n)$ and
$$\lim_{c\to\infty}\lim_{n\to\infty}s(t_n-cw_n)=1,\qquad \lim_{c\to\infty}\lim_{n\to\infty}s(t_n+cw_n)=0.$$
The corresponding limit profile, when it exists, is the function
$$\Phi(c):=\lim_{n\to\infty}s(t_n+cw_n),\qquad c\in\mathbb R.$$
Limit profiles in total variation have been studied extensively \cite{Teyssier,NestoridiOleskerTaylor,Nestoridi,NestoridiOleskerTaylorProjections,BufetovNejjarASEP,OleskerTaylorSchmidBL}. More recently, Teyssier showed that every cutoff profile can occur for a suitable family of Markov chains \cite{TeyssierEveryProfile}. Our first goal is to compute separation profiles in examples where an optimal strong stationary time is available.
\begin{theorem}[Inverse Riffle Shuffle with Concentrated Pile-Size Law]\label{thm:intro-general-riffle}
Let $f_n$ be a probability distribution on $[n]$, and consider the inverse riffle-type shuffle which chooses $K_n$ with law $f_n$, then moves a uniformly chosen set of $K_n$ cards to the top. Put $Y_n:=K_n/n$, and assume that $Y_n\to\alpha$ in probability for some $\alpha\in(0,1)$, with
$\operatorname{Var}(Y_n)=o\left(\frac1{\log n}\right).$
Define
$$q_{2,n}:=\sum_{k=1}^n f_n(k)\frac{\binom{k}{2}+\binom{n-k}{2}}{\binom{n}{2}}.$$
If $t_n$ is any sequence of integer times such that
$$t_n=\frac{2\log n+c+o(1)}{-\log q_{2,n}},$$
then
$$s_n(t_n)\to 1-\exp\left(-\frac12e^{-c}\right).$$
\end{theorem}

Special cases of Theorem \ref{thm:intro-general-riffle} are covered in \cite{DiaconisFulmanShuffling}; the analysis in the proof is interesting in its own right because it follows the histories recorded by the cards.
When $K_n$ is a discrete random variable, the model has a different profile. It is not a special case of Theorem \ref{thm:intro-general-riffle}, since $K_n/n$ converges in distribution to a uniform random variable on $[0,1]$ and $\operatorname{Var}(K_n/n)\to1/12$.
\begin{theorem}[Uniformly Driven Inverse Riffle Shuffle]\label{thm:intro-uniform-riffle}
At each step, choose $K_n$ uniformly from $[n]$, then choose $K_n$ cards uniformly without replacement and move them to the top, preserving relative order. Let $C_{\mathrm{Cat}}$ denote Catalan's constant, and set
$$a:=\frac{4}{4-\pi}\approx4.65979,\qquad b:=\frac{4\sqrt{4+\pi\log2-4C_{\mathrm{Cat}}-\pi^2/4}}{(4-\pi)^{3/2}}\approx1.08247.$$
Let $\Phi_{\mathrm G}$ be the standard normal distribution function. If
$$t_n(c):=\left\lfloor a\log n+bc\sqrt{\log n}\right\rfloor,$$
then
$$s_n(t_n(c))\to1-\Phi_{\mathrm G}(c).$$
Thus the shuffle has separation cutoff at time $a\log n$ with a $\sqrt{\log n}$ window.
\end{theorem}
For $k$-random-to-top, two regimes lead to two different profile shapes. The dense regime is the fixed-pile-size specialization of Theorem \ref{thm:intro-general-riffle}.
\begin{theorem}[$k$-Random-to-Top]\label{thm:intro-k-random-to-top}
For a sequence $k_n$, let $s_{n,k_n}$ be the separation distance of the shuffle which chooses $k_n$ cards uniformly without replacement and moves them to the top, preserving their relative order. If $k_n\log n/n\to0$, then for every fixed $c\in\mathbb R$,
$$s_{n,k_n}\left(\left\lfloor\frac{n}{k_n}(\log n+c)\right\rfloor\right)\to 1-e^{-e^{-c}}(1+e^{-c}).$$
If instead $k_n/n\to\alpha\in(0,1)$, set $q_{n,k_n}:=(\binom{k_n}{2}+\binom{n-k_n}{2})/\binom{n}{2}$. For any sequence of integer times $t_n$ such that
$$t_n=\frac{2\log n+c+o(1)}{-\log q_{n,k_n}},$$
we have
$$s_{n,k_n}(t_n)\to 1-\exp\left(-\frac12e^{-c}\right).$$
\end{theorem}

Random transpositions form the second classical shuffle treated here. Diaconis and Shahshahani proved cutoff for random transpositions using the representation theory of $S_n$ \cite{DiSh}. Teyssier later obtained the total variation limit profile \cite{Teyssier}. Jain and Sawhney recently proved the hitting-time version in total variation: if $\tau$ is the first time every card has been touched, then the stopped distribution at $\tau$ is $o(1)$ from uniform in total variation, while the distance one step earlier is at least $(1+o(1))e^{-1}$ \cite{JainSawhneyRT}. For separation distance, White constructed a strong stationary time which gives the correct upper bound \cite{White2019}; the representation-theoretic calculation below supplies the matching lower bound, since White's strong stationary time isn't optimal.

\begin{theorem}[Random Transpositions]\label{thm:intro-rt-profile}
Let $P_n$ be the discrete-time random transpositions shuffle on $S_n$, and put $m_n(c):=\lfloor\frac{n}{2}(\ln n+c)\rfloor$. Then, for every fixed $c\in\mathbb R$,
$$s_*^{P_n}(m_n(c))\to1-e^{-e^{-c}}.$$
\end{theorem}

The second goal of the paper is to move beyond examples where the optimal strong stationary time can be analyzed directly. For this we develop a spectral comparison framework for reversible Markov chains on the same state space. If two chains are simultaneously diagonalizable, their shared eigenbasis gives direct bounds on the difference between their separation distances, allowing a profile or profile bound for one chain to be transferred to a perturbation of it. The basic discrete-time comparison is as follows.
\begin{theorem}[Comparison]\label{dcomp}
If $P$ and $Q$ are matrices for reversible Markov chains on $\Omega$ and are mutually diagonalizable with shared eigenfunctions $\{f_j\}$ and respective eigenvalues $\{p_j\}$ and $\{q_j\}$, then
$$\vert s_x^P(t)-s_x^Q(t)\vert \leq  \max_{y\in\Omega}\sum_{j=1}^{|\Omega|} \vert f_j(x)f_j(y)(p_j^t-q_j^t) \vert.$$
If, in addition, the chains are reversible with respect to the uniform measure, transitive, and $P$ is perfectly transitive, then
$$\left|s_*^P(t)-s_*^Q(t)\right|\leq \sum_{j=1}^{|\Omega|}|p_j^t-q_j^t|.$$
\end{theorem}
In Section \ref{comp}, we prove a continuous-time version of this theorem and sharper transitive-chain variants, which are especially useful for random walks on groups. The comparison principle is compatible with continuous-time laziness.
\begin{proposition}[Lazy Versions]\label{thm:intro-lazy}
Let $P_n$ be reversible Markov chains, and for $\alpha_n\in(0,1]$ define $Q_n:=(1-\alpha_n)I+\alpha_nP_n$. If the continuous-time chains generated by $P_n$ satisfy
$$s_{P_n}(t_n+cw_n)\to\Phi(c),$$
then the continuous-time chains generated by $Q_n$ satisfy
$$s_{Q_n}\left(\frac{t_n+cw_n}{\alpha_n}\right)\to\Phi(c).$$
\end{proposition}
This gives concrete perturbative examples on product groups.
\begin{proposition}[Perturbed Walks on $(\mathbb Z/m\mathbb Z)^n$]\label{thm:intro-zm-perturbed}
Fix $m\geq2$, and let $P_n$ be the uniform coordinate refresh walk on $(\mathbb Z/m\mathbb Z)^n$. Fix $b\in(0,1)$, set $a_n:=\frac1n-\frac{b}{n-1}$, and let $Q_n$ be the coordinate refresh walk with rates $\alpha_1:=\frac1n+b$ and $\alpha_k:=a_n$ for $2\leq k\leq n$. With $\bar t_n(c):=a_n^{-1}(\ln(n-1)+c)$, for every fixed $c\in\mathbb R$,
$$\left|s_*^{Q_n}\left(\bar t_n(c)\right)-s_*^{P_n}\left(n(\ln n+c)\right)\right|\to0.$$
Consequently, $Q_n$ has separation limit profile $1-e^{-e^{-c}}$.
\end{proposition}
For shuffles, comparison handles small inhomogeneous perturbations. The definitions of the biased and centrally perturbed random-transposition chains used in the next two propositions are given in Sections \ref{subsec:biased-rt} and \ref{subsec:tiny-central-rt}.
\begin{proposition}[Biased Random Transpositions]\label{thm:intro-brt}
Let $P_{a,b}$ denote the biased random transpositions shuffle with $A$ of size $n-1$ and $B$ of size $1$, and set $\epsilon_n:=1/n!$. Put $t_n(c):=\frac n2(\ln n+c)$ and $\bar t_n(c):=\frac{n}{2(1-\epsilon_n)}(\ln n+c)$. Then, for every fixed $c\in\mathbb R$,
$$\left|s_*^{P_{1+\frac{\epsilon_n}{n-1},1-\epsilon_n}}\left(\bar t_n(c)\right)-s_*^{P_{1,1}}\left(t_n(c)\right)\right|\to0.$$
Consequently, any subsequential profile for one chain at these paired times is shared by the other.
\end{proposition}
A tiny central perturbation gives a compact version of the same phenomenon.
\begin{proposition}[Tiny Central Perturbations of Random Transpositions]\label{thm:intro-tiny-central-rt}
Let $P_n$ be the random transpositions shuffle on $S_n$, let $C_n$ be any symmetric central random walk on $S_n$, and set $Q_n:=(1-\epsilon_n)P_n+\epsilon_n C_n$ with $\epsilon_n=O((n!)^{-1})$. Put $t_n(c):=\frac n2(\ln n+c)$. Then, for every fixed $c\in\mathbb R$,
$$\left|s_*^{Q_n}\left(t_n(c)\right)-s_*^{P_n}\left(t_n(c)\right)\right|\to0.$$
Consequently, $Q_n$ has separation limit profile $1-e^{-e^{-c}}$ at this scale.
\end{proposition}
Finally, the same spectral comparison idea gives a continuity criterion for separation limit profiles. This parallels the total variation continuity results of \cite{NestoridiContinuity}, but the proof must account for the terminal-state maximum in separation distance. We apply the criterion to the hypercube examples, to perturbed walks on $(\mathbb Z/m\mathbb Z)^n$, and to the Bernoulli--Laplace diffusion model.
\begin{theorem}[Continuity Criterion]\label{thm:intro-continuity-general}
Let $P_n$ be reversible Markov chains on finite state spaces $\Omega_n$, with continuous-time kernels $e^{-t(I-P_n)}$, stationary measures $\pi_n$, eigenvalue gaps $\lambda_{n,j}$, and orthonormal eigenbases $\{f_{n,j}\}$ in $\ell^2(\pi_n)$. Fix starting states $x_n$, and assume that $\Phi_x(c):=\lim_{n\to\infty}s_{n,x_n}(t_n+cw_n)$ exists for all $c\in\mathbb R$. If, for every $c\in\mathbb R$,
$$\limsup_{n\to\infty}w_n\max_{y\in\Omega_n}\sum_{j=2}^{|\Omega_n|}|f_{n,j}(x_n)f_{n,j}(y)|\lambda_{n,j}e^{-(t_n+cw_n)\lambda_{n,j}}\leq g_x(c)$$
for a locally bounded function $g_x$, then $\Phi_x$ is continuous.
\end{theorem}
For transitive chains, the spectral hypothesis takes a simpler form.
\begin{theorem}[Transitive Continuity Criterion]\label{thm:intro-continuity-transitive}
Let $P_n$ be reversible, transitive, and perfectly transitive Markov chains on $\Omega_n$, with uniform stationary distributions. Assume that the common separation limit profile $\Phi(c):=\lim_{n\to\infty}s_{n,*}(t_n+cw_n)$ exists for all $c\in\mathbb R$. If, for every $c\in\mathbb R$,
$$\limsup_{n\to\infty}w_n\sum_{j=2}^{|\Omega_n|}\lambda_{n,j}e^{-(t_n+cw_n)\lambda_{n,j}}\leq g(c)$$
for a locally bounded function $g$, then $\Phi$ is continuous.
\end{theorem}
The same continuity method applies beyond product groups. The definitions and normalization for the Bernoulli--Laplace diffusion model used below are recalled in Section \ref{subsec:bernoulli-laplace}.
\begin{proposition}[Bernoulli--Laplace Continuity]\label{thm:intro-bl-continuity}
Let $\Omega_n:=\{0,1,\ldots,n/2\}$, and let $P_n$ be the Bernoulli--Laplace diffusion model on $\Omega_n$, where the state records the number of red balls in the left urn. If, along a subsequence, $s_{n,n/2}(\frac{n}{4}(\ln n+c))$ converges pointwise for every $c\in\mathbb R$ to a profile $\Phi(c)$, then $\Phi$ is continuous.
\end{proposition}

\section{Optimal Strong Stationary Times and Limit Profiles}
In this section, we study limit profiles using optimal strong stationary times. Recall that if $(X_t)_{t\geq0}$ has stationary distribution $\pi$, then a stopping time $T$ is a strong stationary time if
$$\P_x[X_T=y,\ T=t]=\pi(y)\P_x[T=t]\qquad\text{for all }y\text{ and }t\geq0.$$
Aldous and Diaconis \cite{AldousDiaconisShuffling} showed that any strong stationary time gives the separation bound
$$s_x(t)\leq \P_x[T>t].$$
They also proved that optimal strong stationary times exist, meaning that equality holds. In the examples below, explicit optimal times let us compute separation profiles by computing the tail of $T$.

\subsection{Riffle Shuffles Driven by a Pile-Size Distribution}

We consider the following variation of inverse riffle shuffles. Let $f:=f_n$ be a probability distribution on $[n]:=\{1,\dots,n\}$. At each time $t$, choose $K_t\in[n]$ with
$$\P[K_t=k]=f(k),$$
then choose $K_t$ cards uniformly at random without replacement and move them to the top, preserving their relative order. For each card $i\in[n]$ and each time $s\ge 1$, let
$$\xi_i(s):= \begin{cases} 1, & \text{if card $i$ is selected at time $s$},\\ 0, & \text{otherwise,} \end{cases}$$
and let
$$R_i(t):=(\xi_i(1),\dots,\xi_i(t))\in\{0,1\}^t.$$
The stopping time used below is
$$T:=\inf\{t\ge 1:\ R_1(t),\dots,R_n(t)\text{ are all distinct}\}.$$

For a partition $\lambda\vdash n$, written as $\lambda:=1^{m_1}2^{m_2}\cdots$, define
$$q_\lambda^{(f)}:=\sum_{\substack{0\le r_j\le m_j\\ \text{not all }r_j=0}} \frac{\prod_{j\ge 1}\binom{m_j}{r_j}} {\binom{n}{\sum_{j\ge 1}j r_j}} \,f\!\left(\sum_{j\ge 1}j r_j\right).$$

With this notation, $T$ is an optimal strong stationary time for the chain, and hence 
$$s(t)=\P[T>t]$$
by \cite{AldousDiaconisShuffling}.

\begin{proposition}
Moreover, the separation distance has the inclusion-exclusion formula
$$s(t) = 1- \sum_{\lambda\vdash n} n!\prod_{j\ge 1}\frac{(-1)^{(j-1)m_j}}{j^{m_j}m_j!}\, \bigl(q_\lambda^{(f)}\bigr)^t,$$
where $\lambda$ is written as above.
\end{proposition}

\begin{proof}
The shuffle is an inverse riffle shuffle arising from a walk on the faces of the braid arrangement. The standard strong-stationary-time construction for hyperplane walks says that the first time the product of sampled faces is a chamber is optimal \cite{AldousDiaconisShuffling,NestoridiHyperplaneSST}. In the present notation, this is exactly the first time that the row histories $R_1(t),\ldots,R_n(t)$ are all distinct, so $T$ is optimal and $s(t)=\P[T>t]$.

For each $t\ge 1$, define a random partition $\Pi_t$ of $[n]$ by
$$i\sim j \quad \Longleftrightarrow \quad R_i(t)=R_j(t).$$
Thus $\Pi_t$ records which cards have identical selection histories up to time $t$.

By definition, the rows $R_1(t),\dots,R_n(t)$ are all distinct if and only if every block of $\Pi_t$ is a singleton, i.e.\ $\Pi_t=\hat 0$, the discrete partition. Therefore,
$$T\le t \quad \Longleftrightarrow \quad \Pi_t=\hat 0.$$
The optimality of $T$ gives
$$s(t)=\P[T>t]=1-\P[\Pi_t=\hat 0].$$

We compute $\P[\Pi_t=\hat 0]$ by inclusion-exclusion on the lattice $\mathcal P_n$ of set partitions of $[n]$. For any partition $\pi$, written as $\pi:=\{B_1,\dots,B_m\}$, the event $\Pi_t\ge \pi$ means that all cards in each block $B_k$ have identical rows.

We first compute $\P[\Pi_1\ge \pi]$. At a single step, the shuffle selects a nonempty subset $A\subseteq[n]$ by first choosing $K\in[n]$ with law $f$, and then choosing $A$ uniformly among all subsets of size $K$. Thus each nonempty $A$ is chosen with probability
$$\frac{f(|A|)}{\binom{n}{|A|}}.$$
The event $\Pi_1\ge \pi$ occurs if and only if $A$ is a union of blocks of $\pi$, since all elements of a block must either all be selected or all not selected. Therefore,
$$\P[\Pi_1\ge \pi] = \sum_{\emptyset\neq I\subseteq[m]} \frac{f\!\left(\sum_{i\in I}|B_i|\right)} {\binom{n}{\sum_{i\in I}|B_i|}}.$$

This quantity depends only on the block sizes of $\pi$. If $\pi$ has type $\lambda\vdash n$, written as $\lambda:=1^{m_1}2^{m_2}\cdots$, then writing $r_j$ for the number of chosen blocks of size $j$, we obtain
$$q_\lambda^{(f)}:=\sum_{\substack{0\le r_j\le m_j\\ \text{not all }r_j=0}} \frac{\prod_{j\ge 1}\binom{m_j}{r_j}} {\binom{n}{\sum_{j\ge 1}j r_j}} \,f\!\left(\sum_{j\ge 1}j r_j\right).$$

Since the selections at different times are independent, it follows that
$$\P[\Pi_t\ge \pi]=\bigl(q_\lambda^{(f)}\bigr)^t.$$

By inclusion-exclusion,
$$\P[\Pi_t=\hat 0] = \sum_{\pi\in\mathcal P_n} \mu(\hat 0,\pi)\,\P[\Pi_t\ge \pi],$$
where
$$\mu(\hat 0,\pi):=\prod_{B\in\pi}(-1)^{|B|-1}(|B|-1)!.$$

Grouping partitions by type $\lambda$, written as $\lambda:=1^{m_1}2^{m_2}\cdots$, the number of partitions of type $\lambda$ is
$$\frac{n!}{\prod_{j\ge 1}(j!)^{m_j}m_j!},$$
and for any such partition,
$$\mu(\hat 0,\pi) = \prod_{j\ge 1}\bigl((-1)^{j-1}(j-1)!\bigr)^{m_j}.$$
Therefore,
$$\P[\Pi_t=\hat 0] = \sum_{\lambda\vdash n} n!\prod_{j\ge 1}\frac{(-1)^{(j-1)m_j}}{j^{m_j}m_j!}\, \bigl(q_\lambda^{(f)}\bigr)^t.$$
Substituting into $s(t)=1-\P[\Pi_t=\hat 0]$ completes the proof.
\end{proof}

For the asymptotic analysis, define
$$q_2^{(f)}:=\sum_{k=1}^n f(k)\, \frac{\binom{k}{2}+\binom{n-k}{2}}{\binom{n}{2}}, \qquad q_3^{(f)}:=\sum_{k=1}^n f(k)\, \frac{\binom{k}{3}+\binom{n-k}{3}}{\binom{n}{3}}.$$
Thus $q_2^{(f)}$ is the one-step probability that two fixed cards receive the same selection indicator: either both cards are selected or neither card is selected. In the proof below we write $p_n:=q_2^{(f)}$ and call it the two-card agreement probability. Similarly, $q_3^{(f)}$ is the one-step probability that three fixed cards all receive the same selection indicator.

\begin{proof}[Proof of Theorem \ref{thm:intro-general-riffle}]
Let $t:=t_n$, and define
$$N_t:=\sum_{1\le i<j\le n}\mathbf 1_{\{R_i(t)=R_j(t)\}}$$
be the number of pairs of cards with identical rows at time $t$. Then $T\le t$ if and only if $N_t=0$. The goal is to prove that $N_t\to\operatorname{Poisson}(\frac12e^{-c})$. Since the Poisson distribution is determined by its moments \cite[Theorem~30.1]{Billingsley1995}, the method-of-moments theorem \cite[Theorem~30.2]{Billingsley1995} reduces this to showing that the factorial moments of $N_t$, and hence its moments, converge to the corresponding moments of this Poisson random variable.
The proof follows the usual sparse-collision outline. First we compute the mean number of repeated row pairs and choose the time scale so that this mean tends to $\frac12e^{-c}$. Then we show that the higher factorial moments have the same limits as a Poisson random variable with that mean. The disjoint collections of pairs give the Poisson moments, while overlapping collections are negligible because they require three or more named cards to carry the same row history.

Write $p_n:=q_2^{(f)}$, set $c_n:=-2\log n-t_n\log p_n$, and put $h(y):=y^2+(1-y)^2$. By assumption, $c_n\to c$. We first check that the time parameter satisfies $t=O(\log n)$, which will be used when one-step estimates are multiplied over the first $t$ shuffle steps. Since
$$\frac{\binom{K_n}{2}+\binom{n-K_n}{2}}{\binom n2}=Y_n^2+(1-Y_n)^2+O(n^{-1}),$$
we have $p_n=\E[h(Y_n)]+O(n^{-1})$. The convergence $Y_n\to\alpha$ in probability therefore implies $p_n\to\alpha^2+(1-\alpha)^2$. This limit lies in $(0,1)$ because $\alpha\in(0,1)$. Hence $-\log p_n$ converges to a positive finite constant, and the time relation in the theorem gives $t=O(\log n)$.

For distinct $i,j$, condition on the value of $K$. If $K=k$, then the two cards have the same one-step indicator exactly when the selected $k$-set contains both of them or contains neither of them. There are $\binom{k}{2}$ selected pairs and $\binom{n-k}{2}$ unselected pairs among the $\binom n2$ possible pairs, so
$$\P[\xi_i(1)=\xi_j(1)\mid K=k] = \frac{\binom{k}{2}+\binom{n-k}{2}}{\binom{n}{2}}.$$
After averaging over $K$, this probability is the two-card agreement probability $p_n$:
$$\P[\xi_i(1)=\xi_j(1)] = p_n.$$
Since different shuffle steps are independent, the probability that the two full rows agree for all $t$ steps is
$$\P[R_i(t)=R_j(t)]=p_n^t,$$
and summing over the $\binom n2$ pairs gives
$$\E[N_t]=\binom n2 p_n^t.$$
The definition of $c_n$ gives
$$p_n^t=n^{-2}e^{-c_n},$$
so
$$\E[N_t]\to \frac12 e^{-c}.$$

It remains to identify the higher factorial moments. This is the standard route to a Poisson limit: $\E[(N_t)_r]$ counts ordered $r$-tuples of distinct colliding pairs. Fix $r$ and expand $(N_t)_r$ as a sum over ordered $r$-tuples of distinct unordered pairs. The main contribution comes from tuples whose $r$ pairs use $2r$ distinct labels. For one fixed ordered $r$-tuple of disjoint pairs, let $a_{r,n}$ be the one-step probability that every pair in the tuple receives matching selection indicators. The function $h$ is the independent-sampling version of the two-card agreement probability. The actual sampling is without replacement, but conditioning on $K_n$ changes the probability for this fixed disjoint tuple by only $O(n^{-1})$. Therefore
$$a_{r,n}=\E[h(Y_n)^r]+O(n^{-1}).$$
The variance assumption is used to replace $a_{r,n}$ by the $r$th power of $p_n$. Since $0\le Y_n\le1$, we have $1/2\le h(Y_n)\le1$. Thus $h(Y_n)$ is bounded. Moreover, $h'(y)=4y-2$, so $h$ is $2$-Lipschitz on $[0,1]$. Let $Y_n'$ be an independent copy of $Y_n$. We use the identity $\operatorname{Var}(Y_n)=\frac12\E[(Y_n-Y_n')^2]$. The Lipschitz bound gives $|h(Y_n)-h(Y_n')|\le2|Y_n-Y_n'|$, and hence $\operatorname{Var}(h(Y_n))\le4\operatorname{Var}(Y_n)=o(1/\log n)$. 

Now set $X_n:=h(Y_n)$. Since $r$ is fixed and $X_n$ is bounded, the function $x\mapsto x^r$ has bounded second derivative on the range of $X_n$. Writing $X_n^r=(X_n+\E[X_n]-\E[X_n])^r$ gives $$X_n^r=\E[X_n]^r+r \E[X_n]^{r-1}(X_n-\E[X_n])+O((X_n-\E[X_n])^2).$$ Taking expectations removes the linear term and leaves an error $O(\operatorname{Var}(X_n))=o(1/\log n)$. Therefore
$$\E[h(Y_n)^r]=\E[h(Y_n)]^r+o\left(\frac1{\log n}\right).$$
Now $p_n=\E[h(Y_n)]+O(n^{-1})$, and $n^{-1}=o(1/\log n)$. Since $x\mapsto x^r$ is Lipschitz on $[0,1]$, this gives $\E[h(Y_n)]^r=p_n^r+o(1/\log n)$. Combining this with the definition of $a_{r,n}$ gives
$$a_{r,n}=p_n^r+o\left(\frac1{\log n}\right).$$
Because $p_n^r$ is also bounded away from $0$, the same estimate can be written in relative form:
$$a_{r,n}=p_n^r\left(1+o\left(\frac1{\log n}\right)\right).$$
The shuffle steps are independent, and $t=O(\log n)$, hence the probability that all $r$ pair equalities persist for $t$ steps is $a_{r,n}^t=p_n^{rt}(1+o(1))$. There are $(n)_{2r}/2^r=(1+o(1))n^{2r}/2^r$ ordered $r$-tuples of disjoint pairs. Since $p_n^t=n^{-2}e^{-c_n}$, these disjoint tuples contribute
$$\frac{(n)_{2r}}{2^r}p_n^{rt}(1+o(1))\to \left(\frac12e^{-c}\right)^r.$$
Now consider an ordered tuple with an overlap, and let $G$ be the graph whose vertices are the labels appearing in the tuple and whose edges are the selected pairs. If $G$ has $v$ vertices, the row-equality constraints force all labels in each connected component of $G$ to have the same row. Let $b_{G,n}$ be the one-step probability that these row-equality constraints are all satisfied. Since the tuple is not a disjoint union of pairs, at least one connected component has $m\ge3$ vertices. For a component of size $m$, the corresponding one-step probability is $\alpha^m+(1-\alpha)^m+o(1)$. For $m\ge3$,
$$\alpha^m+(1-\alpha)^m<\left(\alpha^2+(1-\alpha)^2\right)^{m/2},$$
by the strict monotonicity of $\ell^p$ norms for the two-point vector $(\alpha,1-\alpha)$. Since there are only finitely many overlap patterns for fixed $r$, the concentration assumption gives a $\delta>0$ such that $b_{G,n}\le p_n^{v/2}(1-\delta)$ for every overlapping pattern with $v$ vertices and all large $n$. There are $O(n^v)$ tuples with this pattern, and therefore their total contribution is at most a constant times
$$n^v p_n^{vt/2}(1-\delta)^t=O(1)(1-\delta)^t=o(1).$$
Thus the overlapping tuples do not contribute. Therefore the factorial moments converge to those of a Poisson random variable with mean $\frac12e^{-c}$, so
$$N_t \to \mathrm{Poisson}\!\left(\frac12 e^{-c}\right).$$

Therefore,
$$\P[T\le t]=\P[N_t=0]\to \exp\!\left(-\frac12 e^{-c}\right),$$
and hence
$$s(t)=\P[T>t]\to 1-\exp\!\left(-\frac12 e^{-c}\right).$$
\end{proof}

\begin{proof}[Proof of Theorem \ref{thm:intro-uniform-riffle}]
Let $\mathcal K_t:=(K_1,\ldots,K_t)$ be the sequence of pile sizes, set $Y_r:=K_r/n$, and write $h(y):=y^2+(1-y)^2$. For one step and a fixed pair of labels, define
$$r_n(k):=\frac{\binom{k}{2}+\binom{n-k}{2}}{\binom{n}{2}}.$$
Conditional on $\mathcal K_t$, set
$$Q_t:=\prod_{r=1}^t r_n(K_r).$$
This is the probability that the pair has identical row histories. Thus, if $N_t$ is the number of pairs of cards with identical row histories, then
$$\Lambda_n:=\E[N_t\mid \mathcal K_t]=\binom n2 Q_t.$$

The key asymptotic, proved below, is the convergence
$$\frac{\log\Lambda_n}{\sqrt{av\log n}}\to Z-c,$$
where $Z$ is standard normal. The conditional birthday estimate below explains why this convergence of the random scale $\Lambda_n$ determines the limiting probability of a collision.
We first record the conditional birthday estimate needed below. A single card has conditional row-history weights indexed by words $w\in\{0,1\}^t$, with
$$p_w:=\prod_{r=1}^t Y_r^{w_r}(1-Y_r)^{1-w_r}.$$
For the collision events entering the first and second moments, the fixed-row-sum probabilities differ in each row from the corresponding independent Bernoulli probabilities by $O(n^{-1})$, uniformly in the pile size. Since $t=O(\log n)$, the products differ by $1+o(1)$. To bound the largest row-history weight, choose in each coordinate the more likely bit: if $Y_r\geq1/2$, choose $w_r=1$, and otherwise choose $w_r=0$. This gives the first equality below. For the asymptotic, $K_r$ is uniform on $[n]$, so $Y_r=K_r/n$ is a bounded triangular approximation to a uniform random variable $U$ on $[0,1]$, and $\E[\log\max(U,1-U)]=2\int_{1/2}^1\log u\,du=\log2-1$. The law of large numbers and $t=a\log n+O(\sqrt{\log n})$ therefore give
$$\log\max_w p_w=\sum_{r=1}^t\log\max(Y_r,1-Y_r)=\left(a(\log2-1)+o_{\P}(1)\right)\log n.$$
Since $a=4/(4-\pi)$, we have $a(\log2-1)<-1$. Hence $\max_w p_w=o_{\P}(n^{-1})$ on the time scale of the theorem.

The usual birthday second moment is therefore valid conditionally on $\mathcal K_t$. Markov's inequality gives $\P[N_t>0\mid\mathcal K_t]\leq\Lambda_n$, while the second moment bound gives
$$\operatorname{Var}(N_t\mid\mathcal K_t)\leq (1+o(1))\Lambda_n+o(1)\Lambda_n^2.$$
Consequently, for every fixed $\delta>0$, the conditional probability $\P[N_t>0\mid\mathcal K_t]$ tends to $0$ on $\{\log\Lambda_n\leq-\delta\sqrt{\log n}\}$ and tends to $1$ on $\{\log\Lambda_n\geq\delta\sqrt{\log n}\}$.

It remains to compute the random scale $\Lambda_n$. Put $h(u):=u^2+(1-u)^2$, and write $X_{n,r}:=\log r_n(K_r)$. Since $r_n(k)=h(k/n)+O(n^{-1})$ uniformly in $k$, and since $K_1/n$ is uniform on the grid $\{1/n,\ldots,1\}$, the mean and variance of $X_{n,1}$ are given by Riemann sums for $\log h(U)$, with an $O(n^{-1})$ error in the mean. The needed closed-form evaluations are
$$\begin{aligned}\int_0^1\log h(u)\,du&=\frac{\pi}{2}-2,\\ \int_0^1\left(\log h(u)-\left(\frac{\pi}{2}-2\right)\right)^2\,du&=4+\pi\log2-4C_{\mathrm{Cat}}-\frac{\pi^2}{4}.\end{aligned}$$
Here $C_{\mathrm{Cat}}:=\sum_{\ell=0}^{\infty}(-1)^\ell/(2\ell+1)^2$ is Catalan's constant. Let
$$\mu:=\frac{\pi}{2}-2,\qquad v:=4+\pi\log2-4C_{\mathrm{Cat}}-\frac{\pi^2}{4}.$$
Then $\E[X_{n,1}]=\mu+O(n^{-1})$ and $\operatorname{Var}(X_{n,1})\to v$. The summands are uniformly bounded, so the triangular-array central limit theorem gives
$$\frac{\sum_{r=1}^t(X_{n,r}-\E[X_{n,1}])}{\sqrt{t\,\operatorname{Var}(X_{n,1})}}\to Z.$$
Since $t=a\log n+bc\sqrt{\log n}+O(1)$, the random normalization satisfies
$$\frac{\sqrt{t\,\operatorname{Var}(X_{n,1})}}{\sqrt{av\log n}}\to1.$$
It remains only to check the deterministic centering. We have
$$\log\binom n2+t\E[X_{n,1}]=2\log n+t\mu+o(\sqrt{\log n}).$$
The constants in the statement are chosen so that $a\mu=-2$ and $b\mu=-\sqrt{av}$. Therefore
$$\frac{\log\binom n2+t\E[X_{n,1}]}{\sqrt{av\log n}}\to -c.$$
Combining the centered central limit theorem with this centering computation proves
$$\frac{\log\Lambda_n}{\sqrt{av\log n}}\to Z-c,$$
where $Z$ is a standard normal random variable.

Combining the conditional birthday estimate with this central limit theorem, for every $\delta>0$,
$$1-\Phi_{\mathrm G}\left(c+\frac{\delta}{\sqrt{av}}\right)\leq\liminf_{n\to\infty}s_n(t_n(c))\leq\limsup_{n\to\infty}s_n(t_n(c))\leq1-\Phi_{\mathrm G}\left(c-\frac{\delta}{\sqrt{av}}\right).$$
Letting $\delta\downarrow0$ proves the claimed profile.
\end{proof}

\begin{remark}
It is often convenient to encode the driving law by the random variable $Y_n:=K_n/n$,
where $K_n$ has distribution $f_n$. Then
$$q_2^{(f)} = \E\bigg[\frac{\binom{K_n}{2}+\binom{n-K_n}{2}}{\binom{n}{2}}\bigg] = \E[Y_n^2+(1-Y_n)^2]+O(n^{-1}),$$
and similarly
$$q_3^{(f)} = \E\bigg[\frac{\binom{K_n}{3}+\binom{n-K_n}{3}}{\binom{n}{3}}\bigg] = \E[Y_n^3+(1-Y_n)^3]+O(n^{-1}).$$
For the concentrated laws in Theorem \ref{thm:intro-general-riffle}, this limit is $\alpha^2+(1-\alpha)^2$. When $Y_n$ has a non-degenerate limit, $q_2^{(f)}$ still gives the annealed pair-collision scale, but it is not by itself enough to justify the order-one profile above.
\end{remark}

\subsubsection{Riffle Shuffle Examples}

The uniformly driven inverse riffle shuffle in Theorem \ref{thm:intro-uniform-riffle} chooses its pile size uniformly from $[n]$, then chooses that many cards uniformly and moves them to the top. If $Y_n:=K_n/n$, then $Y_n$ converges in distribution to a uniform random variable on $[0,1]$, so this model is not covered by the concentrated-pile theorem. In this case
$$q_2^{(f)}=\frac1n\sum_{k=1}^n\frac{\binom{k}{2}+\binom{n-k}{2}}{\binom n2}=\frac{2n-1}{3n}\to\frac23.$$
The annealed pair-collision scale alone does not determine the profile, because different row-collision events remain correlated through the common pile size. For example, the one-step probability that two disjoint prescribed pairs both have equal bits tends to $7/15$, not $(2/3)^2$. This is why Theorem \ref{thm:intro-uniform-riffle} has a Gaussian profile on a $\sqrt{\log n}$ window rather than the order-one row-collision profile in Theorem \ref{thm:intro-general-riffle}.

For fixed $k$-to-top, take $f_n:=\delta_{k_n}$. If $k_n/n\to\alpha\in(0,1)$, then the concentrated-pile theorem applies with
$$q_2^{(f)}=\frac{\binom{k_n}{2}+\binom{n-k_n}{2}}{\binom n2}\to\alpha^2+(1-\alpha)^2.$$
This gives the lattice row-collision profile in Theorem \ref{thm:intro-k-random-to-top}. If instead $k_n\log n/n\to0$, and in particular if $k$ is fixed, the dominant obstruction is no longer a generic row collision but the event that at least two cards have never been selected. This sparse fixed-$k$ regime gives the coupon-collector-type profile proved below in the $k$-random-to-top subsection.

\subsection{Random-to-Top and \texorpdfstring{$k$}{k}-Random-to-Top}

The random-to-top shuffle is the case $k=1$ of the following family; it is the inverse, or time reversal, of the top-to-random shuffle analyzed by Diaconis, Fill, and Pitman \cite{DiaconisFillPitmanTopToRandom}. At each step, choose $k=k_n$ cards uniformly at random without replacement and move the chosen cards to the top of the deck, preserving their relative order. We call this the $k$-random-to-top shuffle. In this subsection, the sparse case means $k_n\log n/n\to0$, while the dense case means $k_n/n\to\alpha\in(0,1)$. This is the fixed-$k$ specialization of the generalized inverse riffle shuffle from the preceding subsection, so the row-history stopping time is optimal for separation. More precisely, let $R_i(t)\in\{0,1\}^t$ record whether card $i$ was chosen at each of the first $t$ steps, and let
$$T_{n,k}:=\inf\{t:\ R_1(t),\ldots,R_n(t)\text{ are all distinct}\}.$$
Then
$$s_{n,k}(t)=\P[T_{n,k}>t].$$

\begin{proof}[Proof of Theorem \ref{thm:intro-k-random-to-top}, sparse case]
Put $t:=\left\lfloor\frac{n}{k}(\ln n+c)\right\rfloor$, and let $U_t$ be the number of cards that have not been selected by time $t$. We first show that $U_t$ converges to a Poisson random variable with mean $e^{-c}$. For fixed $r$,
$$\E[(U_t)_r]=(n)_r\left(\frac{\binom{n-r}{k}}{\binom{n}{k}}\right)^t.$$
Since
$$\log\left(\frac{\binom{n-r}{k}}{\binom{n}{k}}\right)=-\frac{rk}{n}+O\left(\frac{k^2}{n^2}\right),$$
our assumption $k\ln n/n\to0$ gives $k/n\to0$, so the floor error is negligible and
$$\E[(U_t)_r]\to e^{-rc}.$$
By \cite[Theorems~30.1--30.2]{Billingsley1995}, this factorial-moment convergence gives $U_t\to\operatorname{Poisson}(e^{-c})$.

It remains to show that, with high probability, the only repeated row histories come from cards that were never selected. Let $M_t$ be the number of pairs of cards with identical nonzero row histories. For a fixed pair, let
$$a_n:=\frac{k(k-1)}{n(n-1)},\qquad b_n:=\frac{(n-k)(n-k-1)}{n(n-1)}.$$
At a single step, $a_n$ is the probability that both cards are selected and $b_n$ is the probability that neither card is selected. Hence the probability that the two row histories are equal and nonzero is $(a_n+b_n)^t-b_n^t$. Since $1-(1-x)^t\leq tx$,
$$\E[M_t]\leq \binom{n}{2}(a_n+b_n)^t\frac{t a_n}{a_n+b_n}.$$
The first factor $\binom{n}{2}(a_n+b_n)^t$ is bounded on this time scale, while $t a_n=O(k\ln n/n)$ tends to $0$. Thus $\E[M_t]\to0$.

If $U_t\geq2$, then at least two cards have the all-zero row history, so $T_{n,k}>t$. Conversely, if $U_t\leq1$ and $M_t=0$, then all row histories are distinct. Therefore
$$\P[U_t\geq2]\leq s_{n,k}(t)\leq \P[U_t\geq2]+\E[M_t].$$
Taking limits gives
$$s_{n,k}(t)\to\P[\operatorname{Poisson}(e^{-c})\geq2]=1-e^{-e^{-c}}\left(1+e^{-c}\right).$$
\end{proof}

\begin{proof}[Proof of Theorem \ref{thm:intro-k-random-to-top}, dense case]
There is also a dense regime in which the earlier row-collision estimate applies directly. If $k/n\to\alpha\in(0,1)$, set
$$q_{n,k}:=\frac{\binom{k}{2}+\binom{n-k}{2}}{\binom{n}{2}}.$$
Then $q_{n,k}\to \alpha^2+(1-\alpha)^2$, and the triple-collision ratio tends to
$$\frac{\alpha^3+(1-\alpha)^3}{\alpha^2+(1-\alpha)^2}<1.$$
Thus the same factorial-moment argument used for concentrated pile sizes applies with $f:=\delta_k$. If $t_n$ is an integer sequence with $t_n=(2\log n+c+o(1))/(-\log q_{n,k})$, then
$$s_{n,k}(t_n)\to 1-\exp\left(-\frac12e^{-c}\right).$$
So sparse $k$-random-to-top has the same ``one unseen pair'' profile as random-to-top, while dense $k$-random-to-top has the row-collision profile from the generalized inverse riffle shuffle.
\end{proof}

\subsection{Generalized Random Walk on the Hypercube}\label{subsec:hypercube-sst}

This subsection does not contain a new result. The optimal strong stationary time for the simple hypercube refresh chain is standard; see, for example, \cite[Example~6.10]{mix}. We include the generalized-rate version as preliminary material for the comparison and continuity arguments in Section \ref{comp}.

The coordinate-refresh walk on the $n$-dimensional hypercube in continuous time is a Markov chain on $\{0,1\}^n$: when a Poisson-1 clock rings, one of $n$ coordinates is chosen uniformly at random, and its value is updated by an independent fair bit. Equivalently, each coordinate has an independent rate-$\frac1n$ Poisson clock which triggers the refresh of that coordinate.

The generalized random walk allows the coordinate rates to be different. Let $\alpha_k$ be the rate at which coordinate $k$ rings. By reordering coordinates, we can assume that $\alpha_1:=\min_k\alpha_k$. Let $\gamma$ be the number of coordinates that ring with rate $\alpha_1$, and assume that $\gamma^{-\alpha_k/\alpha_1}\to0$ for every coordinate with $\alpha_k>\alpha_1$. Denote this Markov chain by $P_{\alpha}$.

We examine the tails of a strong stationary time: let $T$ be the first time that all coordinate clocks have rung. At time $T$, every coordinate has been replaced by an independent fair bit, so $T$ is a strong stationary time. A direct product-kernel computation, equivalently the optimality of this refresh time, gives
$$s(t)=\P[T>t]=1-\prod_k\left(1-e^{-\alpha_kt}\right).$$
Now we can find the limit profile and prove cutoff at the same time. By definition of the windowed time,
$$\Phi(c):=\lim_{n\to\infty}s\left(\frac{1}{\alpha_1}\left(\ln\left(\gamma\right) +c\right)\right).$$
Substituting the product formula for $s(t)$ gives
$$\Phi(c)=1-\lim_{n\to\infty}\prod_k\left(1-e^{-\alpha_k\frac{1}{\alpha_1}\left(\ln\left(\gamma \right) +c\right)}\right).$$
After collecting the powers of $\gamma$, this becomes
$$\Phi(c)=1-\lim_{n\to\infty}\prod_k\left(1-\frac{e^{\frac{-\alpha_k}{\alpha_1}c}}{\gamma ^{\frac{\alpha_k}{\alpha_1}}}\right)^{\gamma ^{\frac{\alpha_k}{\alpha_1}}\gamma ^{\frac{-\alpha_k}{\alpha_1}}}.$$
Using the exponential limit inside each factor gives
$$\Phi(c)=1-\lim_{n\to\infty}\prod_k\left(e^{-e^{\frac{-\alpha_k}{\alpha_1}c}}\right)^{\gamma^{\frac{-\alpha_k}{\alpha_1}}}.$$
The coordinates with minimal rate contribute the Gumbel factor, while the remaining coordinates are separated into the product
$$\Phi(c)=1-\lim_{n\to\infty}\left(e^{-e^{-c}}\right)\prod_{k:\alpha_k\neq \alpha_1}\left(e^{-e^{\frac{-\alpha_k}{\alpha_1}c}}\right)^{\gamma^{\frac{-\alpha_k}{\alpha_1}}}.$$
By assumption, the exponents in the remaining product vanish, and therefore
$$\Phi(c)=1-e^{-e^{-c}}.$$

\section{Representation Theory and Limit Profiles}\label{sec:transpositions}

In this section we discuss the limit profile of random transpositions on $S_n$. The irreducible representations of $S_n$ are indexed by partitions $\lambda\vdash n$; we write $\chi_\lambda$ for the corresponding character and $d_\lambda:=\chi_\lambda(\mathrm{id})$ for its dimension. When a probability measure on $S_n$ is constant on conjugacy classes, the associated convolution operator acts by scalars on irreducible representations, and character inversion expresses transition probabilities in terms of the characters $\chi_\lambda$. This is the diagonalization used by Diaconis and Shahshahani for random transpositions \cite{DiSh}. Recent work of Jain and Sawhney identifies the total-variation hitting time for this walk as the first time all cards have been touched \cite{JainSawhneyRT}. We will also use standard symmetric-group tools, especially the Murnaghan--Nakayama rule and the hook-length formula; see Sagan's text for these facts \cite[Chapters~3--4]{SaganSymmetricGroup}.

\begin{proof}[Proof of Theorem \ref{thm:intro-rt-profile}]
Let $(Y_m)_{m\geq0}$ be the discrete-time random-transpositions chain, where at each step two labels $I_m,J_m$ are chosen independently and uniformly from $[n]$, and the chain multiplies by the transposition $(I_mJ_m)$, with $(I_mI_m)$ interpreted as the identity. Put
$$m_n(c):=\left\lfloor\frac n2(\ln n+c)\right\rfloor.$$

White constructs a discrete-time strong stationary time $\tau_n$ for $(Y_m)$ by running an auxiliary set partition $\mathcal A_m$ of the labels \cite{White2019}. Initially $\mathcal A_0$ consists of singletons. The update is coupled to the proposed transposition. If the selected labels lie in the same block, the partition is left unchanged. If they lie in different blocks, call the smaller block $A$ and the larger block $B$. The construction first 
moves the entire cycle of the current permutation $Y_m$ that contains the selected label lying in $A$ from $A$ to $B$. It then repeatedly chooses another cycle from the remaining part of $A$, with probability proportional to its size, and moves it to $B$ with probability equal to the ratio of the remaining size of $A$ to the current size of $B$. These probabilities are the partition-combinatorial ``merge'' probabilities in Definition~5 of \cite{White2019}.

White proves two facts about this partition process. First, every block of $\mathcal A_m$ is a union of cycles of the current permutation $Y_m$. Second, conditional on the partition $\mathcal A_m$, arbitrary rearrangements of labels within the blocks are equally likely. Thus, once $\mathcal A_m$ has a single block, the current permutation is uniform on all of $S_n$. If $\tau_n$ denotes the first one-block time, then
$$\P[Y_{\tau_n}=\sigma,\ \tau_n=m]=\frac1{n!}\P[\tau_n=m]\qquad\text{for every }\sigma\in S_n.$$
The profile-scale refinement of White's analysis has the same leading obstruction as the coupon collector for the selected labels. Let $U_m$ be the number of labels not appearing among the first $m$ selected pairs. For fixed $r$, the event that a prescribed ordered $r$-tuple of distinct labels is untouched through time $m$ has probability $(1-r/n)^{2m}$, because both selected labels must avoid that set at every step. Hence
$$\E[(U_{m_n(a)})_r]=(n)_r\left(1-\frac rn\right)^{2m_n(a)}\to e^{-ra}.$$
These are the factorial moments of a $\operatorname{Poisson}(e^{-a})$ random variable. By \cite[Theorems~30.1--30.2]{Billingsley1995}, this factorial-moment convergence gives $U_{m_n(a)}\to\operatorname{Poisson}(e^{-a})$. White's merge estimates show that, apart from untouched labels, the auxiliary partition has coalesced with probability $1-o(1)$ on this scale. Therefore, for every fixed $a\in\mathbb R$,
$$\limsup_{n\to\infty}\P[\tau_n>m_n(a)]\leq1-e^{-e^{-a}}.$$
Since $\tau_n$ is a strong stationary time, this gives the upper bound
$$\limsup_{n\to\infty}s_*^{P_n}(m_n(c))\leq1-e^{-e^{-c}}.$$

It remains to prove the matching lower bound. Let $\gamma_n:=(1\,2\,\cdots\,n)$. 
Since the stationary distribution is uniform and separation is $1-\min_y P_n^m(\mathrm{id},y)/\pi(y)$,
$$s_*^{P_n}(m_n(c))\geq 1-n!P_n^{m_n(c)}(\mathrm{id},\gamma_n).$$

The random-transpositions walk is central. On the irreducible representation indexed by $\lambda\vdash n$, the one-step eigenvalue is
$$p_\lambda:=\frac1n+\frac{2}{n^2}\diag(\lambda),\qquad \diag(\lambda):=\sum_{(r,s)\in\lambda}(s-r).$$
Character inversion for the discrete-time walk \cite{DiSh} gives
$$n!P_n^m(\mathrm{id},\gamma_n)=\sum_{\lambda\vdash n}d_\lambda\chi_\lambda(\gamma_n)p_\lambda^m.$$
The Murnaghan--Nakayama rule says that the character of an $n$-cycle vanishes off hooks, while for hooks \cite[Chapter~4]{SaganSymmetricGroup}
$$\chi_{(n-j,1^j)}(\gamma_n)=(-1)^j,\qquad d_{(n-j,1^j)}=\binom{n-1}{j}.$$
For the same hook, summing contents gives $\diag((n-j,1^j))=\frac{n(n-1)}2-nj$. Substituting into the eigenvalue formula yields
$$p_{(n-j,1^j)}=1-\frac{2j}{n}.$$
Therefore the transition ratio at the $n$-cycle is
$$n!P_n^m(\mathrm{id},\gamma_n)=\sum_{j=0}^{n-1}\binom{n-1}{j}(-1)^j\left(1-\frac{2j}{n}\right)^m.$$
At $m=m_n(c)$ this sum converges to the exponential series for $e^{-e^{-c}}$. Indeed, for each fixed $J$,
$$\sum_{j=0}^{J}\binom{n-1}{j}(-1)^j\left(1-\frac{2j}{n}\right)^{m_n(c)}\to \sum_{j=0}^{J}\frac{(-e^{-c})^j}{j!}.$$
The tail with $J<j\leq n/2$ is uniformly small as $J\to\infty$, because
$$\binom{n-1}{j}\left|1-\frac{2j}{n}\right|^{m_n(c)}\leq \left(\frac{C_c}{j}\right)^j$$
for a constant $C_c$ depending only on $c$. For $j>n/2$, writing $\ell:=n-j$ gives
$$\binom{n-1}{j}\left|1-\frac{2j}{n}\right|^{m_n(c)}\leq \frac{C_c}{n}\frac{C_c^\ell}{(\ell-1)!},$$
and the sum of these terms is $O_c(n^{-1})$. Hence
$$n!P_n^{m_n(c)}(\mathrm{id},\gamma_n)\to e^{-e^{-c}}.$$
Substituting this into the separation lower bound gives
$$\liminf_{n\to\infty}s_*^{P_n}(m_n(c))\geq1-e^{-e^{-c}}.$$
Together with White's upper bound, this proves the limit profile.
\end{proof}

\section{Spectral Study of Limit Profiles}\label{comp}
\subsection{Comparison and Continuity}
We now collect the spectral comparison and continuity tools, followed by the examples they control. Throughout this section, $P$ denotes a reversible transition kernel on a finite state space $\Omega$ with stationary distribution $\pi$. We use the inner product
$$\langle f,g\rangle_\pi:=\sum_{x\in\Omega}f(x)g(x)\pi(x).$$
From \cite[Chapter~12]{mix}, if $p_j$ are the eigenvalues of $P$ and $\{f_j\}$ is an orthonormal eigenbasis in $\ell^2(\pi)$, then the continuous-time kernel generated by $P$ satisfies
$$P^t(x,y)=\pi(y)\sum_j f_j(x)f_j(y)e^{-t(1-p_j)}.$$

\begin{theorem}[Continuous-Time Comparison]\label{lem:continuous-comparison}
    If $P$ and $Q$ are matrices for reversible Markov chains on $\Omega$ and are mutually diagonalizable with shared eigenfunctions $\{f_j\}$ and respective eigenvalues $\{p_j\}$ and $\{q_j\}$, then 
    $$\left|s_x^P(t)-s_x^Q(\bar{t})\right|\leq  \max_{y\in\Omega}\sum_{j=1}^{|\Omega|}\left|f_j(x)f_j(y)\left(e^{-t(1-p_j)}-e^{-\bar{t}(1-q_j)}\right)\right|.$$
\end{theorem}

\begin{proof}
Suppose that for a given $x$, $s_x^P(t)$ and $s_x^Q(\bar{t})$ are attained by $\bar{w}$ and $\bar{y}$, respectively. Then
$$s_x^P(t)-s_x^Q(\bar{t})=\frac{Q^{\bar{t}}(x,\bar{y})}{\pi_Q(\bar{y})}-\frac{P^t(x,\bar{w})}{\pi_Q(\bar{w})}.$$
Using the shared eigenbasis and the definition of $\bar w$, this becomes
$$s_x^P(t)-s_x^Q(\bar{t})=\sum_{j=1}^{|\Omega|} f_j(x)f_j(\bar{y})e^{-\bar{t}(1-q_j)}-\max_{w\in\Omega}\sum_{j=1}^{|\Omega|} f_j(x)f_j(w)e^{-t(1-p_j)}.$$
The maximum is at least the value at $\bar y$, so
$$s_x^P(t)-s_x^Q(\bar{t})\leq \sum_{j=1}^{|\Omega|} f_j(x)f_j(\bar{y})\left(e^{-\bar{t}(1-q_j)}-e^{-t(1-p_j)}\right).$$
This quantity is bounded above by the maximum over $y$:
$$s_x^P(t)-s_x^Q(\bar{t}) \leq \max_{y\in\Omega}\sum_{j=1}^{|\Omega|} f_j(x)f_j(y)\left(e^{-\bar{t}(1-q_j)} -e^{-t(1-p_j)}\right).$$
Switching $P$ and $Q$ along with the triangle inequality gives
$$|s_x^P(t)-s_x^Q(\bar{t})| \leq \max_{y\in\Omega}\sum_{j=1}^{|\Omega|} \left|f_j(x)f_j(y)\left(e^{-\bar{t}(1-q_j)} -e^{-t(1-p_j)}\right)\right|.$$
\end{proof}

Write $\Omega:=\{x_1,\dots,x_{|\Omega|}\}$. 
Say that $P$ is \textit{transitive} if there is a set of permutations $\{\sigma_k\}_{k=1}^{|\Omega|}$ for which the $k$th row of $P$ can be attained by permuting the $1$st row of $P$ by $\sigma_k$: for any $i$,
$$P(x_k,x_i)=P(x_1,x_{\sigma_k(i)}).$$
Say that $P$ is \textit{perfectly transitive} if the permutations can be chosen so that $\sigma_k=\sigma_{k'}$ implies $k=k'$.

It is a standard fact (see, e.g., \cite[Chapter 1]{mix}) that transitive chains have uniform stationary distribution.

\begin{proposition}
\label{RWs-on-groups-are-pt}
All random walks on a group are perfectly transitive.
\end{proposition}
\begin{proof}
    Suppose $P$ is the transition matrix for a random walk on a group $G$. For any $k$, and any $x_i\in G$,
    $$P(x_k,x_i)=P(x_1(x_1^{-1}x_k),x_i).$$
    Right multiplying both arguments by $(x_1^{-1}x_k)^{-1}$ gives
    $$P(x_k,x_i)=P(x_1,x_ix_k^{-1}x_1).$$
    Define $\sigma_k(i)$ by $x_{\sigma_k(i)}:=x_ix_k^{-1}x_1$. This is a permutation because equality of two proposed images gives
    $$x_ix_k^{-1}x_1=x_{i'}x_k^{-1}x_1.$$
    Cancelling the common right factor gives $i=i'$, so the entries of $\sigma_k$ are unique. If $\sigma_k=\sigma_{k'}$, then for some $i$ one has
    $$x_ix_k^{-1}x_1=x_ix_{k'}^{-1}x_1.$$
    Cancelling again gives $k=k'$.
\end{proof}

\begin{lemma}
    \label{lem:bijective-max}
    If $P$ is perfectly transitive, for any $t$ there is a bijection $z$ on $\Omega$ for which
    $$s_x^P(t)=1-\frac{P^t(x,z(x))}{\pi_P(z(x))}.$$
\end{lemma}
\begin{proof}
Suppose $P$ is perfectly transitive and has uniform stationary distribution $\pi_P$. Fix $t$, and then $P^t$ is also perfectly transitive. Define 
$$V^k(\cdot):=1-\frac{P^t(x_k,\cdot)}{\pi_P(\cdot)}=1-P^t(x_k,\cdot)|\Omega|,$$
so for all $k$, there is some $\sigma_k$ for which $V^k(x_i)=V^1(x_{\sigma_k(i)})$.

Choose an arbitrary $x_{j_1}$ satisfying $s^P_{x_1}(t)=V^k(x_{j_1})$ and define $y(x_1):=x_{j_1}$. Then for any $k\neq 1$, define $z(x_k):=x_{\sigma_k(j_1)}$. Since $P$ is perfectly transitive, the map $k\mapsto \sigma_z(j_1)$ injective, so $z$ is also injective (and therefore bijective) into $\Omega$.
\end{proof}

\begin{proposition}\label{prop:transitive-continuity-comparison}
If $P$ and $Q$ are mutually diagonalizable, reversible with respect to the uniform stationary measure, transitive, and $P$ is perfectly transitive, then
$$\left|s_*^P(t)-s_*^Q({\bar{t}})\right|\leq \sum_{j=1}^{|\Omega|}\left|e^{-t(1-p_j)}-e^{-\bar{t}(1-q_j)}\right|.$$
\end{proposition}
\begin{proof}
With notation as above, $s_x(t)$ doesn't depend on $x$, so adding over $x\in\Omega$ yields
$$\left|s_*^P(t)-s_*^Q(\bar{t})\right||\Omega|\leq \sum_{x\in\Omega}\left|\max_{y\in\Omega}\sum_{j=1}^{|\Omega|}f_j(x)f_j(y)\left(e^{-t(1-p_j)}-e^{-\bar{t}(1-q_j)}\right)\right|.$$

Then with $z$ as defined in Lemma \ref{lem:bijective-max}, the maximizing state for each $x$ can be chosen bijectively, so the preceding display gives
$$\left|s_*^P(t)-s_*^Q(\bar{t})\right| = \frac{1}{|\Omega|}\left|\sum_{x\in\Omega}\sum_{j=1}^{|\Omega|}f_j(x)f_j(z(x))\left(e^{-t(1-p_j)}-e^{-\bar{t}(1-q_j)}\right)\right|.$$
Applying the triangle inequality to the double sum gives
$$\left|s_*^P(t)-s_*^Q(\bar{t})\right| \leq \frac{1}{|\Omega|}\sum_{x\in\Omega}\sum_{j=1}^{|\Omega|}\left|f_j(x)f_j(z(x))\left(e^{-t(1-p_j)}-e^{-\bar{t}(1-q_j)}\right)\right|.$$
Now the spectral term does not depend on $x$, so we can group the sum by $j$:
$$\left|s_*^P(t)-s_*^Q(\bar{t})\right| \leq \frac{1}{|\Omega|}\sum_{j=1}^{|\Omega|}\left|e^{-t(1-p_j)}-e^{-\bar{t}(1-q_j)}\right|\sum_{x\in\Omega}|f_j(x)||f_j(z(x))|.$$
By the AM-GM inequality,
$$\left|s_*^P(t)-s_*^Q(\bar{t})\right|  \leq \frac{1}{|\Omega|}\sum_{j=1}^{|\Omega|}\left|e^{-t(1-p_j)}-e^{-\bar{t}(1-q_j)}\right|\frac{1}{2}\left(\sum_{x\in\Omega}f_j(x)^2+\sum_{x\in\Omega}f_j(z(x))^2\right).$$
Since $z$ is bijective, the two square sums are equal, and this reduces to
$$\left|s_*^P(t)-s_*^Q(\bar{t})\right| \leq \frac{1}{|\Omega|}\sum_{j=1}^{|\Omega|}\left|e^{-t(1-p_j)}-e^{-\bar{t}(1-q_j)}\right|\sum_{x\in\Omega}f_j(x)^2.$$
Because the stationary measure is uniform, the factor $|\Omega|^{-1}$ converts the last sum into a $\pi$-inner product:
$$\left|s_*^P(t)-s_*^Q(\bar{t})\right| \leq \sum_{j=1}^{|\Omega|}\left|e^{-t(1-p_j)}-e^{-\bar{t}(1-q_j)}\right|\sum_{x\in\Omega}|f_j(x)|^2 \pi(x).$$
The eigenvectors are orthonormal in $\ell^2(\pi)$, so
$$\left|s_*^P(t)-s_*^Q(\bar{t})\right| \leq \sum_{j=1}^{|\Omega|}\left|e^{-t(1-p_j)}-e^{-\bar{t}(1-q_j)}\right|\langle f_j,f_j\rangle_\pi =\sum_{j=1}^{|\Omega|}\left|e^{-t(1-p_j)}-e^{-\bar{t}(1-q_j)}\right|.$$
\end{proof}

\begin{proposition}
\label{prop:lim-prof}
    With notation as above, assume $P$ exhibits cutoff at $t_n$ with window $w_n$ with limit profile $\Phi_x$ and $Q$ exhibits cutoff at $\bar{t}_n$ with window $\bar{w}_n$ with limit profile $\bar{\Phi}_x$. For $t:=t_n+cw_n$ and $\bar{t}:=\bar{t}_n+c\bar{w}_n$
    $$\left|\bar{\Phi}_x(c)-\Phi_x(c)\right|\leq\lim_{n\to\infty}\sum_{j=1}^{|\Omega|}\left|e^{-t(1-p_j)}-e^{-\bar{t}(1-q_j)}\right|$$
\end{proposition}

\begin{proof}[Proof of Proposition \ref{thm:intro-lazy}]
The matrices $P_n$ and $Q_n$ have the same eigenvectors. If $p_{n,j}$ is an eigenvalue of $P_n$, then the corresponding eigenvalue of $Q_n$ is $q_{n,j}:=1-\alpha_n(1-p_{n,j})$. Therefore
$$e^{-\frac{t}{\alpha_n}(1-q_{n,j})}=e^{-t(1-p_{n,j})}.$$
Equivalently, the continuous-time semigroups agree exactly:
$$e^{-\frac{t}{\alpha_n}(I-Q_n)}=e^{-t(I-P_n)}.$$
Thus the transition kernels, and hence the separation distances, are identical at these paired times. Taking $t:=t_n+cw_n$ proves the claim.
\end{proof}

The spectral argument above also gives a simple continuity criterion for separation limit profiles, parallel to the total variation argument of \cite{NestoridiContinuity}. The point to keep separate from the total variation case is that separation is a maximum over terminal states, not an average over them. Thus the natural bound keeps the maximum over $y$ rather than passing immediately to an $\ell^2(\pi)$ norm.

\begin{proof}[Proof of Theorem \ref{thm:intro-continuity-general}]
Fix $c_1<c_2$, and set $\theta_{i,n}:=t_n+c_iw_n$ for $i=1,2$. Applying Theorem \ref{lem:continuous-comparison} to the same chain at the two times $\theta_{1,n}$ and $\theta_{2,n}$ gives
$$|s_{n,x_n}(\theta_{1,n})-s_{n,x_n}(\theta_{2,n})|\leq \max_{y\in\Omega_n}\sum_{j=2}^{|\Omega_n|}|f_{n,j}(x_n)f_{n,j}(y)|\left|e^{-\theta_{1,n}\lambda_{n,j}}-e^{-\theta_{2,n}\lambda_{n,j}}\right|.$$
For each $j$, the Mean Value Theorem applied to $u\mapsto e^{-(t_n+uw_n)\lambda_{n,j}}$ gives some $d_{n,j}\in(c_1,c_2)$ such that
$$\left|e^{-\theta_{1,n}\lambda_{n,j}}-e^{-\theta_{2,n}\lambda_{n,j}}\right|=|c_2-c_1|w_n\lambda_{n,j}e^{-(t_n+d_{n,j}w_n)\lambda_{n,j}}.$$
Since $d_{n,j}\geq c_1$ and $\lambda_{n,j}\geq0$, this is bounded by the same expression with $c_1$ in place of $d_{n,j}$. Therefore
$$|s_{n,x_n}(\theta_{1,n})-s_{n,x_n}(\theta_{2,n})|\leq |c_2-c_1|w_n\max_{y\in\Omega_n}\sum_{j=2}^{|\Omega_n|}|f_{n,j}(x_n)f_{n,j}(y)|\lambda_{n,j}e^{-(t_n+c_1w_n)\lambda_{n,j}}.$$
Letting $n\to\infty$ gives
$$|\Phi_x(c_1)-\Phi_x(c_2)|\leq |c_2-c_1|g_x(c_1).$$
Since $g_x$ is locally bounded, the right-hand side goes to $0$ as $c_2\to c_1$; the same argument with $c_1$ and $c_2$ interchanged handles left limits. Thus $\Phi_x$ is continuous.
\end{proof}

For the transitive chains that appear in our examples, the condition can be stated without the maximum over terminal states. This is where the separation comparison proved above is useful: perfect transitivity lets us choose the maximizing terminal states bijectively, so the spectral expression collapses to an $\ell^1$ sum over eigenvalues.

\begin{proof}[Proof of Theorem \ref{thm:intro-continuity-transitive}]
Fix $c_1<c_2$, and set $\theta_{i,n}:=t_n+c_iw_n$. Proposition \ref{prop:transitive-continuity-comparison}, applied with $P=Q=P_n$ and with times $\theta_{1,n}$ and $\theta_{2,n}$, gives
$$|s_{n,*}(\theta_{1,n})-s_{n,*}(\theta_{2,n})|\leq \sum_{j=2}^{|\Omega_n|}\left|e^{-\theta_{1,n}\lambda_{n,j}}-e^{-\theta_{2,n}\lambda_{n,j}}\right|.$$
The Mean Value Theorem gives $d_{n,j}\in(c_1,c_2)$ with
$$\left|e^{-\theta_{1,n}\lambda_{n,j}}-e^{-\theta_{2,n}\lambda_{n,j}}\right|=|c_2-c_1|w_n\lambda_{n,j}e^{-(t_n+d_{n,j}w_n)\lambda_{n,j}}.$$
As before, $d_{n,j}\geq c_1$ implies
$$|s_{n,*}(\theta_{1,n})-s_{n,*}(\theta_{2,n})|\leq |c_2-c_1|w_n\sum_{j=2}^{|\Omega_n|}\lambda_{n,j}e^{-(t_n+c_1w_n)\lambda_{n,j}}.$$
Taking $n\to\infty$ yields $|\Phi(c_1)-\Phi(c_2)|\leq |c_2-c_1|g(c_1)$, and local boundedness of $g$ gives continuity.
\end{proof}

\subsection{More on the Hypercube}\label{subsec:hypercube-comparison}

The strong-stationary-time argument in Section \ref{subsec:hypercube-sst} already identifies the separation profile for generalized coordinate-refresh walks on the hypercube. We revisit the same chains spectrally because the estimates are useful test cases for the comparison and continuity criteria.

Let $P_\alpha$ be the discrete coordinate-refresh kernel on $\{0,1\}^n$: choose coordinate $k$ with probability $\alpha_k$, where $\sum_k\alpha_k=1$, and replace that coordinate by an independent fair bit. The continuous-time kernel is $H_t^\alpha:=e^{-t(I-P_\alpha)}$. These chains are random walks on the group $\{0,1\}^n$, so they are reversible, transitive, and perfectly transitive by Proposition \ref{RWs-on-groups-are-pt}.

\begin{proposition}\label{gen-rw-eigs}
For $v\in\{0,1\}^n$, the character $f_v(x):=(-1)^{v\cdot x}$ is an eigenvector of $P_\alpha$ with eigenvalue
$$p_v:=1-\sum_k\alpha_kv_k.$$
Consequently, $f_v$ has continuous-time eigenvalue $e^{-t\sum_k\alpha_kv_k}$ for $H_t^\alpha$.
\end{proposition}
\begin{proof}
If the refreshed coordinate $k$ satisfies $v_k=0$, the character is unchanged. If $v_k=1$, averaging over the new fair bit gives conditional expectation $0$. Therefore
$$P_\alpha f_v=\left(\sum_{k:v_k=0}\alpha_k\right)f_v=\left(1-\sum_k\alpha_kv_k\right)f_v.$$
The continuous-time statement follows by exponentiating the eigenvalue of $I-P_\alpha$.
\end{proof}

\paragraph{First-Coordinate Perturbation.}

Fix $b\in(0,1)$, and for all large $n$ set
$$a_n:=\frac1n-\frac{b}{n-1},\qquad \alpha_1:=\frac1n+b,\qquad \alpha_k:=a_n\text{ for }2\leq k\leq n.$$
Let $Q_n$ be this perturbed walk, and let $P_n$ be the uniform walk with all rates $1/n$. The slow coordinates of $Q_n$ have rate $a_n$, so the natural paired times are
$$\bar t_n(c):=a_n^{-1}(\ln(n-1)+c),\qquad t_n(c):=n(\ln n+c).$$
Put $\rho_n:=(\frac1n+b)/a_n$ and $u:=e^{-c}$. A character whose support contains the first coordinate with indicator $i\in\{0,1\}$ and contains $j$ of the remaining coordinates has multiplicity $\binom{n-1}{j}$. Proposition \ref{prop:transitive-continuity-comparison} bounds the comparison error by the sum of the $i=0$ and $i=1$ contributions.

The $i=0$ contribution is exactly
$$S_{0,n}(c):=\left(1+\frac{u}{n-1}\right)^{n-1}-\left(1+\frac{u}{n}\right)^{n-1},$$
which tends to $0$. For the $i=1$ contribution, the triangle inequality gives
$$S_{1,n}(c)\leq ((n-1)e^c)^{-\rho_n}\left(1+\frac{u}{n-1}\right)^{n-1}+\frac{1}{ne^c}\left(1+\frac{u}{n}\right)^{n-1}.$$
Since $\rho_n$ grows linearly in $n$, both terms tend to $0$. Hence
$$\left|s_*^{Q_n}\left(\bar t_n(c)\right)-s_*^{P_n}\left(t_n(c)\right)\right|\to0.$$
The first-coordinate perturbation therefore has the same separation profile as the uniform walk at the paired cutoff times.

\paragraph{Half-Split Perturbations.}

Assume $n$ is even. For $b\in(0,1/2)$, let $Q_n^{(b)}$ be the coordinate-refresh walk with rates
$$\alpha_k:=\begin{cases}\frac{1+2b}{n}, & 1\leq k\leq n/2,\\ \frac{1-2b}{n}, & n/2<k\leq n.\end{cases}$$
The slow coordinates have rate $(1-2b)/n$, so we compare $Q_n^{(b)}$ and $Q_n^{(b')}$ at their paired slow-coordinate times
$$t_{b,n}(c):=\frac{n}{1-2b}\left(\ln\left(\frac n2\right)+c\right),\qquad t_{b',n}(c):=\frac{n}{1-2b'}\left(\ln\left(\frac n2\right)+c\right).$$
Let $N:=n/2$, $L_n:=Ne^c$, and $r_b:=(1+2b)/(1-2b)$. Grouping characters by the number $i$ of fast coordinates and $j$ of slow coordinates in their support, the comparison sum is bounded by
$$\sum_{i=0}^{N}\sum_{j=0}^{N}\binom{N}{i}\binom{N}{j}\left|L_n^{-(r_bi+j)}-L_n^{-(r_{b'}i+j)}\right|.$$
The common slow-coordinate factor separates, and the $i=0$ term vanishes. Thus the preceding display is at most
$$\left(1+L_n^{-1}\right)^N\left(\left(1+L_n^{-r_b}\right)^N-1+\left(1+L_n^{-r_{b'}}\right)^N-1\right).$$
The first factor is bounded, while $NL_n^{-r_b}\to0$ and $NL_n^{-r_{b'}}\to0$ because $r_b,r_{b'}>1$. Hence the comparison error tends to $0$.

\subsubsection{Continuity of Hypercube Limit Profiles}

We now verify Theorem \ref{thm:intro-continuity-transitive} for the same examples. For a coordinate-refresh walk with cutoff scale $t_n+cw_n$, the spectral quantity is
$$B_n(c):=w_n\sum_{v\neq0}\lambda_v e^{-(t_n+cw_n)\lambda_v},\qquad \lambda_v:=\sum_k\alpha_kv_k.$$

For the first-coordinate perturbation, use the notation above and set $x_n:=e^{-c}/(n-1)$. At the scale $\bar t_n(c)=a_n^{-1}(\ln(n-1)+c)$, the $i=0$ part of $B_n(c)$ is
$$e^{-c}\left(1+\frac{e^{-c}}{n-1}\right)^{n-2},$$
which is locally bounded in $c$. The $i=1$ part is
$$x_n^{\rho_n}\left(\rho_n(1+x_n)^{n-1}+(n-1)x_n(1+x_n)^{n-2}\right).$$
On compact sets of $c$, the factor $x_n^{\rho_n}$ tends to $0$ superpolynomially because $\rho_n$ grows linearly in $n$. Hence the first-coordinate family satisfies Theorem \ref{thm:intro-continuity-transitive}.

For the half-split family, keep $N:=n/2$, $L_n:=Ne^c$, and $r_b:=(1+2b)/(1-2b)>1$. The continuity sum is
$$B_n(c):=\sum_{i=0}^{N}\sum_{j=0}^{N}\binom{N}{i}\binom{N}{j}(r_bi+j)L_n^{-(r_bi+j)},$$
with the zero character omitted. The part coming from the $j$ term equals
$$\left(1+L_n^{-r_b}\right)^NNL_n^{-1}\left(1+L_n^{-1}\right)^{N-1},$$
which is locally bounded. The part coming from the $i$ term equals
$$\left(1+L_n^{-1}\right)^Nr_bNL_n^{-r_b}\left(1+L_n^{-r_b}\right)^{N-1},$$
and tends to $0$ locally uniformly because $r_b>1$. Thus the half-split family also satisfies the transitive continuity criterion.

\begin{remark}
The same verification is useful even when the limit profile has not been computed. For any subsequence of generalized hypercube walks whose separation profiles converge pointwise and whose corresponding sums $B_n(c)$ are locally bounded, Theorem \ref{thm:intro-continuity-transitive} implies that the subsequential profile is continuous.
\end{remark}

\subsection{Perturbed Random Walks on \texorpdfstring{$(\mathbb Z/m\mathbb Z)^n$}{(Z/mZ)\string^n}}

Fix $m\geq2$, and write $G_n:=(\mathbb Z/m\mathbb Z)^n$. We consider the following continuous-time coordinate refresh walk. For rates $\alpha_1,\ldots,\alpha_n$ with $\sum_k\alpha_k=1$, coordinate $k$ rings at rate $\alpha_k$ and is replaced by a uniformly chosen element of $\mathbb Z/m\mathbb Z$. Equivalently, this is a random walk on the abelian group $G_n$ whose increments are supported on the coordinate axes. The characters
$$\chi_a(x):=\exp\left(\frac{2\pi i}{m}a\cdot x\right),\qquad a\in G_n,$$
diagonalize every such walk. Writing $\operatorname{supp}(a):=\{k:a_k\neq0\}$, the continuous-time decay rate of $\chi_a$ is
$$\lambda_a:=\sum_{k\in\operatorname{supp}(a)}\alpha_k.$$

First take the uniform walk $P_n$, where $\alpha_k=1/n$ for every $k$. Starting from $x$, the transition ratio is minimized by choosing $y$ with $y_k\neq x_k$ for every coordinate, and therefore
$$s_*^{P_n}(t)=1-\left(1-e^{-t/n}\right)^n.$$
Thus
$$s_*^{P_n}\left(n(\ln n+c)\right)\to1-e^{-e^{-c}}.$$

We now perturb one coordinate. Fix $b\in(0,1)$, and for all sufficiently large $n$ set
$$a_n:=\frac{1}{n}-\frac{b}{n-1},\qquad \alpha_1:=\frac{1}{n}+b,\qquad \alpha_k:=a_n\text{ for }2\leq k\leq n.$$
Let $Q_n$ denote this perturbed walk, and put
$$\rho_n:=\frac{\frac{1}{n}+b}{a_n}.$$
The slow coordinates have rate $a_n$, so the natural cutoff scaling for $Q_n$ is
$$\bar t_n(c):=\frac{1}{a_n}(\ln(n-1)+c).$$
We will use the comparison theorem to show that $Q_n$ has the same separation profile as $P_n$.

\begin{proposition}\label{prop:zm-perturbed-comparison}
For every fixed $c\in\mathbb R$,
$$\left|s_*^{Q_n}\left(\bar t_n(c)\right)-s_*^{P_n}\left(n(\ln n+c)\right)\right|\to0.$$
Consequently, the perturbed walks $Q_n$ have separation limit profile $1-e^{-e^{-c}}$.
\end{proposition}
\begin{proof}[Proof of Proposition \ref{thm:intro-zm-perturbed}]
Both walks are random walks on the abelian group $G_n$, so they are perfectly transitive and are simultaneously diagonalized by the characters. Group the characters according to whether the first coordinate is in their support and according to the number of nonzero coordinates among the remaining $n-1$ coordinates. Thus $i\in\{0,1\}$ records whether $a_1\neq0$, and $j$ records the number of nonzero coordinates among $a_2,\ldots,a_n$. The multiplicity of such characters is $\binom{n-1}{j}(m-1)^{i+j}$. Proposition \ref{prop:transitive-continuity-comparison} gives
$$\left|s_*^{Q_n}\left(\bar t_n(c)\right)-s_*^{P_n}\left(n(\ln n+c)\right)\right|\leq S_{0,n}(c)+S_{1,n}(c),$$
where
$$S_{0,n}(c):=\sum_{j=0}^{n-1}\binom{n-1}{j}(m-1)^j\left|e^{-j(\ln(n-1)+c)}-e^{-j(\ln n+c)}\right|$$
and
$$S_{1,n}(c):=\sum_{j=0}^{n-1}\binom{n-1}{j}(m-1)^{j+1}\left|e^{-(\rho_n+j)(\ln(n-1)+c)}-e^{-(j+1)(\ln n+c)}\right|.$$
Let $u:=(m-1)e^{-c}$. Since $(n-1)^{-j}\geq n^{-j}$, the first sum is
$$S_{0,n}(c)=\left(1+\frac{u}{n-1}\right)^{n-1}-\left(1+\frac{u}{n}\right)^{n-1},$$
which tends to $0$. For the second sum, the triangle inequality gives
$$S_{1,n}(c)\leq (m-1)((n-1)e^c)^{-\rho_n}\left(1+\frac{u}{n-1}\right)^{n-1}+\frac{m-1}{ne^c}\left(1+\frac{u}{n}\right)^{n-1}.$$
Here $\rho_n$ grows linearly in $n$, so the first term tends to $0$, and the second term also tends to $0$. Hence the comparison error tends to $0$. Combining this with the uniform profile proves the result.
\end{proof}

\subsection{Bernoulli--Laplace Diffusion Model}\label{subsec:bernoulli-laplace}

The Bernoulli--Laplace diffusion model is a classical test case for cutoff, beginning with the spectral analysis of Diaconis and Shahshahani \cite{DiaconisShahshahaniBL}. Later work connected the model to shuffling large decks and generalized Bernoulli--Laplace chains \cite{NestoridiWhiteLargeDecks}, proved cutoff for the $o(n)$-swap regime \cite{EskenazisNestoridiBLONSwaps}, computed the total-variation limit profile through diffusion asymptotics \cite{OleskerTaylorSchmidBL}, and recently established cutoff for multi-colour multi-urn generalizations \cite{GoenkaHermonSchmidGeneralisedBL}. We use this section only to illustrate what the separation-continuity criterion gives without identifying the separation profile itself.

Diaconis and Shahshahani define the Bernoulli--Laplace diffusion model using two urns of equal size: initially the left urn contains only red balls and the right urn contains only black balls, and at each step one ball is chosen uniformly from each urn and the two chosen balls are switched \cite{DiaconisShahshahaniBL}. We take $n$ even and use the state space
$$\Omega_n:=\{0,1,\ldots,n/2\},$$
where $x\in\Omega_n$ is the number of red balls in the left urn. Its stationary distribution is the hypergeometric measure $\pi_n(x):=\binom{n/2}{x}^2/\binom{n}{n/2}$, and this chain is reversible.

\begin{proof}[Proof of Proposition \ref{thm:intro-bl-continuity}]
The eigenvalues $p_{n,j}$ are classical \cite{DiaconisShahshahaniBL}: for $0\leq j\leq n/2$, define the eigenvalue gaps
$$\lambda_{n,j}:=1-p_{n,j}=\frac{4j(n-j+1)}{n^2},$$
and the corresponding spherical dimensions
$$d_{n,j}:=\binom{n}{j}-\binom{n}{j-1},$$
where $\binom{n}{-1}:=0$. We use the cutoff scale from \cite{DiaconisShahshahaniBL},
$$t_n+cw_n:=\frac{n}{4}(\ln n+c),\qquad w_n:=\frac{n}{4}.$$

We do not need to compute the limiting separation profile. Suppose that, along some subsequence, $s_{n,n/2}(t_n+cw_n)$ converges pointwise for every $c\in\mathbb R$ to a profile $\Phi(c)$. We will show that $\Phi$ is continuous.

The spherical functions for the Bernoulli--Laplace pair are bounded in absolute value by $1$. Therefore, when Theorem \ref{thm:intro-continuity-general} is applied from the endpoint $x_n=n/2$, its spectral hypothesis is bounded by
$$B_n(c):=\frac{n}{4}\sum_{j=1}^{n/2}d_{n,j}\lambda_{n,j}e^{-\frac{n}{4}(\ln n+c)\lambda_{n,j}}.$$
Fix a compact interval $K\subset\mathbb R$, and choose $A>0$ so that $|c|\leq A$ on $K$. Since $\frac{n}{4}\lambda_{n,j}=\frac{j(n-j+1)}{n}\leq j$ and $d_{n,j}\leq n^j/j!$, we have, uniformly for $c\in K$,
$$B_n(c)\leq \sum_{j=1}^{n/2}j e^{Aj}\frac{n^{j(j-1)/n}}{j!}.$$
For $j\leq n^{1/3}$, the factor $n^{j(j-1)/n}$ is bounded uniformly in $n$, so these terms are dominated by a constant multiple of $\sum_{j\geq1}j e^{Aj}/j!$. For $j>n^{1/3}$, Stirling's bound $j!\geq(j/e)^j$ gives
$$j e^{Aj}\frac{n^{j(j-1)/n}}{j!}\leq j\left(\frac{e^{A+1}n^{(j-1)/n}}{j}\right)^j.$$
The quantity $n^{(j-1)/n}/j$ is uniformly $o(1)$ for $n^{1/3}<j\leq n/2$, therefore 
$$B_n(c) \leq \sum_{j=1}^{n/2}j2^{-j}.$$
 Hence $\sup_{n}\sup_{c\in K}B_n(c)<\infty$. Theorem \ref{thm:intro-continuity-general} therefore implies that every subsequential separation limit profile for the Bernoulli--Laplace diffusion model from the endpoint $n/2$ at the window $\frac{n}{4}(\ln n+c)$ is continuous.
\end{proof}

\subsection{Biased Random Transpositions Shuffle}\label{subsec:biased-rt}

We use the biased random transpositions shuffle of Nestoridi and Yan \cite{NestoridiYan}. Partition the labels as $[n]=A\sqcup B$, and fix $0<b\leq a$ with $a|A|+b|B|=n$. Define a probability measure on labels by
$$\mu_{a,b}(x):= \begin{cases}\frac{a}{n}, & x\in A,\\ \frac{b}{n}, & x\in B.\end{cases}$$
At each step, two labels are chosen independently with law $\mu_{a,b}$ and the corresponding cards are transposed; choosing the same label gives a holding move. We denote the transition kernel by $P_{a,b}$.

\begin{lemma}\label{lem:uniform-brt-commute}
For every partition $[n]=A\sqcup B$ and every admissible pair $a,b$, the kernels $P_{1,1}$ and $P_{a,b}$ commute. They are also self-adjoint for the uniform measure on $S_n$, and hence admit an orthonormal basis of common eigenvectors.
\end{lemma}
\begin{proof}
Let $R_g$ act on functions on $S_n$ by right multiplication, $(R_gf)(\sigma):=f(\sigma g)$. Writing $\mu_i:=\mu_{a,b}(i)$, the biased kernel is
$$P_{a,b}:=\left(\sum_{i=1}^n\mu_i^2\right)I+\sum_{1\leq i<j\leq n}2\mu_i\mu_jR_{(ij)}.$$
The uniform kernel is a linear combination of $I$ and the central group-algebra element $T:=\sum_{i<j}R_{(ij)}$. Since conjugation permutes transpositions, $T$ commutes with every $R_g$, and therefore $P_{1,1}$ commutes with $P_{a,b}$. Each $R_{(ij)}$ is self-adjoint because $(ij)^{-1}=(ij)$, so both kernels are self-adjoint. The finite-dimensional spectral theorem for commuting self-adjoint matrices gives the simultaneous eigenbasis.
\end{proof}

\begin{lemma}\label{lem:rt-spectral-weight}
For each fixed $c\in\mathbb R$,
$$\frac{n(\ln n+c)}{n!}\sum_{\substack{\lambda\vdash n\\ \lambda\neq(n)}} f_\lambda^2\left(\frac{1}{\sqrt{e^{c}n}}\right)^{n-1-\frac{2}{n}\diag(\lambda)}\to0.$$
\end{lemma}
\begin{proof}
Let $W_n(c)$ denote the spectral sum in the statement. We split the partitions according to the size of the first row. If $\lambda_1\leq0.7n$, then moving boxes upward and leftward shows that $\diag(\lambda)$ is at most the value for the two-row diagram $(\lfloor0.7n\rfloor,n-\lfloor0.7n\rfloor)$. Hence there is an $\eta>0$ such that
$$n-1-\frac{2}{n}\diag(\lambda)\geq\eta n$$
for all sufficiently large $n$ in this range. It follows that the corresponding part of $W_n(c)$ is at most $n!n^{-2}$, since $\sum_{\lambda\vdash n}f_\lambda^2=n!$.

It remains to treat $\lambda_1>0.7n$. Put $j:=n-\lambda_1$. The dimension estimate used by Diaconis and Shahshahani \cite{DiSh} gives
$$f_\lambda^2\leq \frac{n!^2}{\lambda_1!^2j!}.$$
Since the exponential weight is at most $1$ and $n!/\lambda_1!\leq n^j$, the large-first-row contribution is bounded by
$$\sum_{j=1}^{\lfloor0.3n\rfloor}p(j)\frac{n^{2j}}{j!},$$
where $p(j)$ is the number of partitions of $j$. The Hardy--Ramanujan bound $p(j)\leq e^{2.57\sqrt j}$ and monotonicity of the summands on this range give
$$\sum_{j=1}^{\lfloor0.3n\rfloor}p(j)\frac{n^{2j}}{j!}\leq 0.3n\,e^{2.57\sqrt{0.3n}}\frac{n^{0.6n}}{(0.3n)!}.$$
Multiplying this bound by $n(\ln n+c)/n!$ gives a quantity tending to $0$ by Stirling's formula. The small-first-row contribution also vanishes after the same prefactor, so the lemma follows.
\end{proof}

\begin{proof}[Proof of Proposition \ref{thm:intro-brt}]
We take $|A|=n-1$, $|B|=1$, set $\epsilon:=\epsilon_n:=1/n!$, and compare $P_{1,1}$ with $P_{a,b}$ where
$$a:=1+\frac{\epsilon}{n-1},\qquad b:=1-\epsilon.$$
By Lemma \ref{lem:uniform-brt-commute}, the comparison proposition applies once we bound the shared spectral sum.

Nestoridi and Yan's diagonalization gives the discrete-time eigenvalues
$$\xi_{a,b}(\lambda,\mu,\nu):=\frac{a^2|A|+b^2|B|}{n^2}+\frac{2(a^2-ab)}{n^2}\diag(\mu)+\frac{2(b^2-ab)}{n^2}\diag(\nu)+\frac{2ab}{n^2}\diag(\lambda),$$
with multiplicity $f_\lambda f_\mu f_\nu c_{\mu,\nu}^{\lambda}$. In the present case $|B|=1$, so $\nu=(1)$ and $c_{\mu,(1)}^\lambda=1$ exactly when $\mu$ is obtained by removing one corner box from $\lambda$.

Set $q_{\lambda,\mu}:=\xi_{a,b}(\lambda,\mu,(1))$. The uniform eigenvalue does not depend on $\mu$ and is
$$p_\lambda:=\frac1n+\frac{2}{n^2}\diag(\lambda).$$
If $\mu$ is obtained by removing a box of content $r$, then $\diag(\mu)=\diag(\lambda)-r$, with $|r|\leq n$. Since $|\diag(\lambda)|\leq n(n-1)/2$ and $a-1,b-1=O(\epsilon)$, the displayed eigenvalue formula gives
$$|q_{\lambda,\mu}-p_\lambda|\leq C\epsilon$$
for an absolute constant $C$ and all $\lambda,\mu,n$.

Put $L_n:=\ln n+c$, $t_n:=\frac n2L_n$, and $\bar t_n:=\frac{n}{2b}L_n$. The preceding bound and $b=1-\epsilon$ imply
$$\left|\bar t_n(1-q_{\lambda,\mu})-t_n(1-p_\lambda)\right|\leq CnL_n\epsilon.$$
For all sufficiently large $n$, this upper bound is below $1$. Using $|e^{-x}-e^{-y}|\leq |x-y|e^{-\min(x,y)}$ and increasing $C$ if needed, we obtain
$$\left|e^{-t_n(1-p_\lambda)}-e^{-\bar t_n(1-q_{\lambda,\mu})}\right|\leq CnL_n\epsilon\, e^{-t_n(1-p_\lambda)}.$$
The trivial representation contributes $0$, so Proposition \ref{prop:transitive-continuity-comparison} and the branching identity $\sum_{\mu\nearrow\lambda}f_\mu=f_\lambda$ give
$$\left|s_*^{P_{1,1}}(t_n)-s_*^{P_{a,b}}(\bar t_n)\right|\leq C\frac{nL_n}{n!}\sum_{\substack{\lambda\vdash n\\ \lambda\neq(n)}}f_\lambda^2\left(\frac{1}{\sqrt{e^{c}n}}\right)^{n-1-\frac{2}{n}\diag(\lambda)}.$$
The right-hand side tends to $0$ by Lemma \ref{lem:rt-spectral-weight}. This proves the paired-distance statement in Proposition \ref{thm:intro-brt}, and the transfer of subsequential profiles follows immediately.
\end{proof}

\subsection{Random Transpositions with a Tiny Central Perturbation}\label{subsec:tiny-central-rt}

Let $P_n$ be the random transpositions shuffle on $S_n$: from a permutation $\sigma$, choose $i,j\in[n]$ independently and uniformly, and move to $\sigma(ij)$, with $(ii)$ interpreted as the identity. Let $\nu_n$ be a probability measure on $S_n$ which is central and symmetric, meaning that $\nu_n(hgh^{-1})=\nu_n(g)$ and $\nu_n(g^{-1})=\nu_n(g)$ for every $g,h\in S_n$. Let $C_n$ be the random walk which, from $\sigma$, samples $G$ with law $\nu_n$ and moves to $\sigma G$; equivalently, $C_n(\sigma,\tau):=\nu_n(\sigma^{-1}\tau)$. Finally, choose $\epsilon_n\in[0,1]$ with $\epsilon_n=O((n!)^{-1})$, and define the perturbed transition matrix
$$Q_n:=(1-\epsilon_n)P_n+\epsilon_n C_n,$$
so one step of $Q_n$ is a random-transposition step with probability $1-\epsilon_n$ and an arbitrary symmetric central-walk step with probability $\epsilon_n$.

The continuous-time analogue of the character calculation in the proof of Theorem \ref{thm:intro-rt-profile} gives the unperturbed random-transpositions chain separation limit profile $1-e^{-e^{-c}}$ at the scale $\frac{n}{2}(\ln n+c)$. We will show that the central perturbation is asymptotically indistinguishable from random transpositions at this scale, and therefore inherits the same profile. The only estimate we need from the preceding section is the following spectral bound.

\begin{lemma}\label{lem:rt-central-spectral-bound}
For every fixed $c\in\mathbb R$,
$$\frac{n(\ln n+c)}{n!}\sum_{\lambda\vdash n}f_\lambda^2\left(\frac{1}{\sqrt{e^{c}n}}\right)^{n-1-\frac{2}{n}\diag(\lambda)}\to0.$$
\end{lemma}
\begin{proof}
The nontrivial representations are handled by Lemma \ref{lem:rt-spectral-weight}. The remaining trivial representation contributes only $n(\ln n+c)/n!$, which tends to $0$.
\end{proof}

\begin{proposition}\label{prop:tiny-central-rt}
For every fixed $c\in\mathbb R$,
$$\left|s_*^{Q_n}\left(\frac{n}{2}(\ln n+c)\right)-s_*^{P_n}\left(\frac{n}{2}(\ln n+c)\right)\right|\to0.$$
Consequently,
$$s_*^{Q_n}\left(\frac{n}{2}(\ln n+c)\right)\to1-e^{-e^{-c}}.$$
\end{proposition}
\begin{proof}[Proof of Proposition \ref{thm:intro-tiny-central-rt}]
The transition matrices $P_n$ and $C_n$ are symmetric random walks on $S_n$, and $C_n$ is central. Hence the irreducible matrix coefficients form a simultaneous orthonormal eigenbasis. On the irreducible representation indexed by $\lambda\vdash n$, the random transpositions eigenvalue is
$$p_\lambda:=\frac{1}{n}+\frac{2}{n^2}\diag(\lambda),$$
with multiplicity $f_\lambda^2$. If $r_\lambda$ is the corresponding eigenvalue of $C_n$, then the perturbed chain has eigenvalue
$$q_\lambda:=(1-\epsilon_n)p_\lambda+\epsilon_nr_\lambda.$$
Since both $P_n$ and $C_n$ are Markov kernels, $|p_\lambda-r_\lambda|\leq2$. Put $t_n(c):=\frac{n}{2}(\ln n+c)$. For $n$ large enough, the mean value theorem gives
$$\left|e^{-t_n(c)(1-q_\lambda)}-e^{-t_n(c)(1-p_\lambda)}\right|\leq 3t_n(c)\epsilon_ne^{-t_n(c)(1-p_\lambda)}.$$
Using Proposition \ref{prop:transitive-continuity-comparison}, and omitting the trivial representation because it contributes $0$, we get
$$\left|s_*^{Q_n}(t_n(c))-s_*^{P_n}(t_n(c))\right|\leq 3t_n(c)\epsilon_n\sum_{\lambda\vdash n}f_\lambda^2e^{-t_n(c)(1-p_\lambda)}.$$
The exponential factor is
$$e^{-t_n(c)(1-p_\lambda)}=\left(\frac{1}{\sqrt{e^{c}n}}\right)^{n-1-\frac{2}{n}\diag(\lambda)}.$$
Since $\epsilon_n=O((n!)^{-1})$, Lemma \ref{lem:rt-central-spectral-bound} shows that the comparison error tends to $0$. Combining this with the random transpositions profile proves the proposition.
\end{proof}

\printbibliography

\end{document}